\documentclass{amsart}
\title[Recursive formulas for sums of squares...]{Recursive formulas for sums of squares and sums of triangular numbers}
\usepackage{amssymb,amsmath,amsthm,epsfig,graphics,latexsym}
 \theoremstyle{definition}
 \newtheorem{definition}{Definition}
  \theoremstyle{plain}
  \newtheorem{lemma}      [definition]{Lemma}
  
  \newtheorem{theorem}    {Theorem}
  \newtheorem{corollary}  {Corollary}

  \theoremstyle{remark}

\begin{document}
  \author{Mohamed El Bachraoui}
  \address{Dept. Math. Sci,
 United Arab Emirates University, PO Box 17551, Al-Ain, UAE}
 \email{melbachraoui@uaeu.ac.ae}
 %\email{msalim@uaeu.ac.ae}
 \keywords{Sums of squares; Sums of triangular numbers; Divisor functions; Inductive formulas; Infinite products; Arithmetic functions;}
 \subjclass{11B75, 11B50}
 %\thanks{Corresponding author: M. El Bachraoui, E-mail: melbachraoui@uaeu.ac.ae}
  %Department of Mathematical Sciences\\
  %United Arab Emirates University,
  %PO Box 17551, Al-Ain, UAE \\
  %\texttt{melbachraoui@uaeu.ac.ae, \ msalim@uaeu.ac.ae}}
  %
  %\date{\textit{\today}}
  %\maketitle
  \begin{abstract}
  We prove recursive formulas for sums of squares and sums of triangular numbers in terms of sums of divisors functions and we
  give a variety of consequences of these formulas. Intermediate applications include statements about positivity
  of the coefficients of some infinite products.
  \end{abstract}
  \date{\textit{\today}}
  \maketitle
 \section{Introduction}
 \noindent
 The purpose of this paper is to give inductive formulas for sums of squares, sums of triangular numbers,
 and mixed sums of squares and triangular numbers. Our main tool is the following theorem.
 Let $\mathbb{N}=\{1,2,3,\ldots\}$ and let $\mathbb{N}_0 = \{0,1,2,\ldots\}$.
 %$\mathbb{Z}= \{\ldots,-3,-2,-1,0,1,2,3,\ldots\}$.
  %
 \begin{theorem}\label{main1} \cite{Elbachraoui}
 Let $\alpha \in \mathbb{N}$. Let $A_1, A_2, \ldots A_{\alpha} \subseteq {\mathbb{N}}$ and
 $\underline{A}=(A_1,A_2,\ldots,A_{\alpha})$ and let
 $f_i : A_i \to \mathbb{C}$ for $i=1,2,\ldots,\alpha$ be arithmetic functions and let
 $\underline{f}=(f_1,f_2,\ldots,f_{\alpha})$.
 If both
 \[
  F_{\underline{A}} (x) =
  \prod_{i=1}^{\alpha} \prod_{n \in A_i}(1- x^{n})^{-\frac{f_i(n)}{n}} =
  \sum_{n=0}^{\infty} p_{\underline{A},\underline{f}}(n) x^n
 \]
 and
 \[
  \sum_{i=1}^{\alpha} \sum_{n \in A_i} \frac{f_i(n)}{n} x^{n}
 \]
 converge absolutely and represent analytic functions in the unit disk $|x|<1$,
 then
 \begin{equation*}\label{main-formula}
 n p_{\underline{A},\underline{f}}(n) =
 \sum_{k=1}^n \left(p_{\underline{A},\underline{f}}(n-k) \sum_{i=1}^{\alpha} f_{i,A_i}(k)\right),
 \end{equation*}
 where
  $p_{\underline{A},\underline{f}}(0)=1$ and
  \[ f_{i,A_i}(k) = \sum_{\substack{d \mid k \\ d\in A_i}}f_i(d). \]
 \end{theorem}
 \noindent
 We note that Theorem~\ref{main1} for the special case
 $\alpha = 1$, has been given in Apostol \cite{Apostol} and in Robbins in \cite{Robbins} to give formulas relating
 arithmetic functions to sums of divisors functions. The authors' key argument is that generating functions for
 the arithmetic functions they considered
 have the form of infinite products ranging over a single set of natural numbers.
 In our previous work \cite{Elbachraoui} we used Theorem~\ref{main1} to deal
 with arithmetic functions whose generating functions
 involve finitely many infinite products ranging over different sets of natural numbers. We derived a variety of
 inductive formulas for such functions.
 In the present note we shall continue employing
 Theorem~\ref{main1} to deduce inductive formulas for sums of squares, sums of triangular numbers, and mixed sums of
 squares and triangular numbers, see Section~2. We notice at this point that there is a large literature on
 sums of squares and sums of triangular numbers. The interested reader is referred to the classical volumes of 
 Dickson~\cite{Dickson} and the recent book of Williams~\cite{Williams} along with their references. 
 With an essential help of identities given in Williams in~\cite{Williams},
  we will give a variety of consequences of our recursive formulas, see Section~3. Further, we shall
 prove that the coefficients of some infinite products are all positive, see Section~4.
 By way of example, we will prove the following results:\\
 (1)
 \[
 \left(\sum_{n=1}\left(\sigma(n)-4\sigma(n/4) \right) x^n \right)
 \left( \sum_{n=1}\left(\sigma^{\ast}(n)-4\sigma^{\ast}(n/2) \right) x^n \right) =
 \]
 \[
 \sum_{n=1}^{\infty} \big( n\left(\sigma(n)-4\sigma(n/4)\right) -
  \left(\sigma^{\ast} (n)- 4\sigma^{\ast}(n/2)\right) \big) x^n,
  \]
  see Corollary~\ref{infinite-sums} below.
  Here $\sigma(n)$ and $\sigma^{\ast}(n)$ are as in Definitions \ref{def:sigma} and \ref{def:sigma-ast} below.
  \\
  (2) If $p\equiv 1\bmod 4$ is prime, then
  \[
 \sum_{j=1}^{p-1} r_2(j)\left( \sigma^{\ast}(p-j)- 4\sigma^{\ast}(\frac{p-j}{2}) \right)= p-1,
 \]
 refer to Theorem~\ref{prime-2-squares} below.
 Here $r_2(n)$ is the number of representations of $n$ as a sum of two squares.
 \\
 (3)
 If $p$ and $4p +1$ are primes, then
 \[
 \sum_{j=1}^{p-1} t_2(j)\left( \sigma(p-j)-4 \sigma(\frac{p-j}{2}) \right) = -1,
 \]
 see Theorem~\ref{p-4p+1} below.
 Here $t_2(n)$ is the number of representations of $n$ as a sum of two triangular numbers. 
  \\
 (4) The following is a direct consequence of Theorem~\ref{master-series} below.
 If $x \in \mathbb{C}$ with $|x|<1$ and $a\in\mathbb{N}$, then the coefficients of the infinite series 
 \[ \prod_{n=1}^{\infty} (1- x^n)^{-a} (1-x^{3n})^{a},\quad 
 \prod_{n=1}^{\infty} (1- x^n)^{-a} (1-x^{3n-1})^{a},
 \]
 and
 \[
  \prod_{n=1}^{\infty} (1- x^n)^{-a} (1-x^{3n})^{a} (1-x^{3n-1})^{a}
 \]
 are all positive.
 
 Our identities involve sums of divisors functions which are introduced in the following three definitions.
  \begin{definition}\label{def:sigma}
 Let the function $\sigma$ be defined on $\mathbb{Q}$ as follows: $\sigma(0)=1$, if $q\in \mathbb{Q}\setminus \mathbb{N}_0$,
 then $\sigma(q)=0$ , and if
 $q\in \mathbb{N}$, then
 \[
 \sigma(q) =
 \sum_{d\mid q}d.
 \]
 \end{definition}
 \begin{definition}
 Let $n, r\in \mathbb{N}_0$, let $m \in \mathbb{N}$, and let
 \[
 \sigma_{r,m}(n) = \sum_{\substack{d\mid n\\ d\equiv r\bmod m}} d
 \]
 If $m=2$ and $r=1$ we often write $\sigma^{o}(n)$ rather than $\sigma_{1,2}(n)$ and
 if $m=2$ and $r=0$ we often write $\sigma^{E}(n)$ rather than $\sigma_{0,2}(n)$.
 \end{definition}
 \noindent
 We have the following basic facts: for $m, n\in \mathbb{N}$
 \begin{equation}\label{divisors-basics}
 \sigma(n)=\sigma^o(n) + \sigma^E(n),\quad \sigma_{m,2m}(n) = m \sigma^o (n/m),\quad\text{and\ } \sigma_{0,m} (n) = m\sigma(n/m).
 \end{equation}
 \begin{definition}\label{def:sigma-ast}
  Let the function $\sigma^{\ast}$ be defined on $\mathbb{Q}$ as follows: if $q\in \mathbb{Q}\setminus \mathbb{N}$,
  then $\sigma^{\ast}(q)=0$ and if
 $q\in \mathbb{N}$, then
 \[
 \sigma^{\ast}(q) =
 \sum_{d\mid q,\ \frac{q}{d}\ \text{odd}}d.
 \]
 \end{definition}
 \noindent
 By \cite[Theorem 3.4]{Williams} we have for all $n\in\mathbb{N}$ that
 \begin{equation}\label{sigma-star}
 \sigma^{\ast}(n) = \sigma(n) - \sigma(n/2).
 \end{equation}
 Further we will need the following two identities due to Jacobi and Gauss respectively:
 \begin{equation} \label{Jacobi}
  \prod_{n=1}^{\infty}(1-x^{2n})(1+x^{2n-1})^2 = \sum_{n=-\infty}^{\infty} x^{n^2}.
  \end{equation}
  \begin{equation} \label{Gauss}
  \prod_{n=1}^{\infty}(1-x^{2n})(1-x^{2n-1})^{-1} = \sum_{n=0}^{\infty} x^{\frac{n(n+1)}{2}}.
  \end{equation}
  Identities (\ref{Jacobi}) and (\ref{Gauss}) can be found for instance in Hardy and Wright \cite{Hardy-Wright}.
 \section{General formulas}
 In this section we give inductive formulas for sums of squares, sums of triangular numbers, and mixed sums of
 squares and triangular numbers.
 \begin{definition}
 Let $k\in\mathbb{N}_0$ and
 let the function $r_k(n)$ be defined on $\mathbb{N}_0$ as follows:
 \[
 r_k(n) = \#\{(x_1,x_2,\ldots,x_k)\in\mathbb{Z}^k:\ n=x_{1}^2+x_{2}^2+\ldots + x_{k}^2 \}.
 \]
 \end{definition}
 \begin{theorem} \label{sum-of-squares}
 If $n \in \mathbb{N}$, then
 \[
 n r_k(n) = 2k \left( \sigma^{\ast}(n)- 4 \sigma^{\ast}(n/2) +
    \sum_{j=1}^{n-1} r_k (j) \left( \sigma^{\ast}(n-j)- 4 \sigma^{\ast}(\frac{n-j}{2}) \right) \right).
 \]
 \end{theorem}
 \begin{proof}
 Taking $k$ powers in identity (\ref{Jacobi}) we have
 \[
 \prod_{n=1}^{\infty}(1-x^{2n})^k (1+ x^{2n-1})^{2k} =
 \left( \sum_{n=-\infty}^{\infty} x^{n^2}\right)^k =
 1 + \sum_{n=1}^{\infty} r_k(n) x^n,
 \]
 or equivalently,
 \[
 \prod_{n=1}^{\infty}(1-x^{2n})^k (1-x^{4n-2})^{2k} (1-x^{2n-1})^{-2k} = 1 + \sum_{n=1}^{\infty} r_k(n) x^n.
 \]
 Let in Theorem~\ref{main1}, $A_1 = 2\mathbb{N}$ and $f_1(n)= -kn$, $A_2= 4\mathbb{N}-2$ and $f_2(n)=-2kn$, and
 $A_3= 2\mathbb{N}-1$ and $f_3(n)= 2kn$. Then for $n\geq 1$
 \begin{equation}\label{eq:squares}
 \begin{split}
  n r_k(n) &= -k \sigma^E (n)- 2k \sigma_{2,4}(n) + 2k \sigma^o (n) + \\
   & \sum_{j=1}^{n-1} r_k(j) \left( -k \sigma^E (n-j)- 2k \sigma_{2,4}(n-j) + 2k \sigma^o (n-j) \right).
  \end{split}
 \end{equation}
 Identities (\ref{divisors-basics}) and formula (\ref{sigma-star}) yield
 \[
 \begin{split}
 - \sigma^E (n)- 2 \sigma_{2,4}(n) + 2 \sigma^o (n) 
 &=
 2\left( \sigma(n)- \sigma(n/2) \right) - 8 \left( \sigma(n/2) - \sigma(n/4) \right) \\
 &=
 2 \sigma^{\ast}(n) - 8 \sigma^{\ast} (n/2).
 \end{split}
 \]
 Now put in formula (\ref{eq:squares}) to conclude the desired identity.
 \end{proof}
 \begin{definition}
 Let $k\in \mathbb{N}_0$ and let the function $t_k$ be defined on $\mathbb{N}_0$ as follows:
 \[
 t_k (n) = \# \left\{ (x_1,x_2,\ldots,x_k)\in \mathbb{N}_{0}^k:\ n =
 \sum_{i=1}^k \frac{x_i(x_i +1)}{2} \right\}.
 \]
 \end{definition}
 \begin{theorem} \label{sum-of-triangulars}
 If $n \in \mathbb{N}_0$, then
 \[
 n t_k(n) = k \sum_{j=0}^{n-1} t_k (j) \left( \sigma(n-j)- 4 \sigma(\frac{n-j}{2}) \right).
 \]
 \end{theorem}
 \begin{proof}
 Taking $k$ powers in identity (\ref{Gauss}) we have
 \[
 \prod_{n=1}^{\infty}(1-x^{2n})^k (1- x^{2n-1})^{-k} =
 \left( \sum_{n=0}^{\infty} x^{\frac{n(n+1)}{2}}\right)^k =
 \sum_{n=0}^{\infty} t_k(n) x^n,
 \]
 and the result follows by
 Theorem~\ref{main1} applied to $A_1 = 2\mathbb{N}$, $f_1(n)= -kn$, $A_2= 2\mathbb{N}-1$, and $f_2(n)=kn$.
 \end{proof}
 \begin{definition}
 Let $k,l \in \mathbb{N}$ and let the function $u_{k,l}$ be defined on $\mathbb{N}_0$ as follows:
 \[
 u_{k,l}(n) = \# \left\{(x_1,\ldots,x_k,y_1,\ldots,y_l)\in \mathbb{Z}^k \times \mathbb{N}_{0}^l:\
  n = \sum_{i=1}^k x_i^2 + \sum_{i=1}^l \frac{y_i(y_i+1)}{2} \right\}.
 \]
 \end{definition}
 \begin{theorem} \label{mixed-sums}
 If $k,l, n\in\mathbb{N}$, then
 \[
 \begin{split}
 n u_{k,l}(n) &= (2k+l)\sigma(n) - 2(5k +2l) \sigma(n/2) + 8k \sigma(n/4) + \\
 & \sum_{j=1}^{n-1} u_{k,l}(j)\left( (2k+l)\sigma(n-j) -2(5k+2l)\sigma(\frac{n-j}{2}) + 8k\sigma(\frac{n-j}{4}) \right).
 \end{split}
 \]
 \end{theorem}
 \begin{proof}
 Clearly
 \[
 \prod_{n=1}^{\infty} (1-x^{2n})^{k+l} (1+x^{2n-1})^{2k} (1-x^{2n-1})^{-l} = 1+ \sum_{n=1}^{\infty} u_{k,l}(n)x^n,
 \]
 that is,
 \[
 \prod_{n=1}^{\infty} (1-x^{2n})^{k+l} (1-x^{4n-2})^{2k} (1-x^{2n-1})^{-2k}(1-x^{2n-1})^{-l} = 1+ \sum_{n=1}^{\infty} u_{k,l}(n)x^n,
 \]
 equivalently,
 \[
 \prod_{n=1}^{\infty} (1-x^{2n})^{3k+l} (1-x^{4n})^{-2k} (1-x^{2n-1})^{-2k-l} = 1+ \sum_{n=1}^{\infty} u_{k,l}(n)x^n,
 \]
 which by Theorem~\ref{main1} gives the desired identity.
 \end{proof}
 \section{Sums of two, four, and eight squares}
 The following formulas for $r_2(n)$, $r_4(n)$, and $r_8(n)$ are well-known and can be found for instance in Williams~\cite{Williams}.
 If $n\in\mathbb{N}$, then
 \begin{equation} \label{r2-r4-r8}
 \begin{split}
 r_2(n) &= 4 \sum_{d\mid n} \left(\frac{-4}{d} \right), \\
 r_4(n) &= 8 \sigma(n) -32 \sigma(n/4), \\
 r_8(n) &= 16 (-1)^n \sum_{d\mid n} (-1)^d d^3,
 \end{split}
 \end{equation}
 where
 \[
 \left(\frac{-4}{d} \right) =
 \begin{cases}
 1,\ \text{if\ } d\equiv 1\bmod 4, \\
 -1,\ \text{if\ } d\equiv 3\bmod 4, \\
 0,\ \text{otherwise.}
 \end{cases}
 \]
 We start by evaluating a convolution.
 \begin{theorem}\label{convolution}
 If $n\in \mathbb{N}$, then
 \[
 \sum_{j=1}^{n-1}\left(\sigma(j)-4\sigma(\frac{j}{4}) \right)
 \left(\sigma^{\ast}(n-j)-4\sigma^{\ast}(\frac{n-j}{2}) \right)=
 \]
 \[
 n \left(\sigma(n)-4 \sigma(\frac{n}{4}) \right)- \left( \sigma^{\ast}(n)-4\sigma^{\ast}(\frac{n}{2}) \right).
 \]
 \end{theorem}
 \begin{proof}
 By Theorem~\ref{sum-of-squares}
 \begin{equation} \label{four-squares}
 n r_4(n) = 8\sigma^{\ast}(n)- 32 \sigma^{\ast}(n/2) +
     \sum_{j=1}^{n-1} r_4 (j) \left( 8 \sigma^{\ast}(n-j)-
    32 \sigma^{\ast}(\frac{n-j}{2}) \right),
 \end{equation}
 which combined with (\ref{r2-r4-r8}) for $r_4(n)$ yields the result.
 \end{proof}
 \begin{corollary} \label{infinite-sums}
 \[
 \left(\sum_{n=1}^{\infty}\left(\sigma(n)-4\sigma(n/4) \right) x^n \right)
 \left( \sum_{n=1}^{\infty}\left(\sigma^{\ast}(n)-4\sigma^{\ast}(n/2) \right) x^n \right) =
 \]
 \[
 \sum_{n=1}^{\infty} \big( n \sigma(n)-4n \sigma(n/4) -
  \sigma^{\ast} (n)+ 4\sigma^{\ast}(n/2) \big) x^n.
  \]
 \end{corollary}
 \begin{proof}
 Immediate from Theorem~\ref{convolution}.
 \end{proof}
 \begin{theorem} \label{prime-2-squares}
 Let $p$ be prime.\\
 (a)\ If $p\equiv 1\bmod 4$, then
 \[
 \sum_{j=1}^{p-1} r_2(j)\left( \sigma^{\ast}(p-j)- 4\sigma^{\ast}(\frac{p-j}{2}) \right)= p-1.
 \]
 (b)\ If $p\equiv 3\bmod 4$, then
 \[
 \sum_{j=1}^{p-1} r_2(j)\left( \sigma^{\ast}(p-j)-4\sigma^{\ast}(\frac{p-j}{2}) \right)= -p -1.
 \]
 \end{theorem}
 \begin{proof}
 by Theorem~\ref{sum-of-squares} we have
 \begin{equation} \label{two-squares}
 n r_2(n) = 4\sigma^{\ast}(n)- 16 \sigma^{\ast}(n/2) +
     \sum_{j=1}^{n-1} r_2 (j) \left( 4 \sigma^{\ast}(n-j)- 16 \sigma^{\ast}(\frac{n-j}{2}) \right).
 \end{equation}
 Further, if $p$ is prime, then clearly $\sigma^{\ast}(p)=1+p$ and $\sigma^{\ast}(p/2)=0$. Moreover, by virtue of
 formulas (\ref{r2-r4-r8}) we have that $r_2(p)=0$
 for $p\equiv 3\bmod 4$ and $r_2(p)=8$ for $p\equiv 1\bmod 4$. Now combine these facts with identity~(\ref{two-squares}) to conclude parts (a) and (b).
 \end{proof}
 \begin{corollary} \label{twins}
 Let $(p,p+2)$ be a twin-prime.
 \\
 (a)\ If $p\equiv 1\bmod 4$, then
 \[
 \sum_{j=1}^{p+1} r_2(j)\left( \sigma^{\ast}(p+2-j)- 4\sigma^{\ast}(\frac{p+2-j}{2}) \right)=
 \]
 \[
 -4 - \sum_{j=1}^{p-1} r_2(j)\left( \sigma^{\ast}(p-j)- 4\sigma^{\ast}(\frac{p-j}{2}) \right).
 \]
 (b)\ If $p\equiv 3\bmod 4$, then
 \[
 \sum_{j=1}^{p+1} r_2(j)\left( \sigma^{\ast}(p+2-j)-4\sigma^{\ast}(\frac{p+2-j}{2}) \right)=
 \]
 \[
 - \sum_{j=1}^{p-1} r_2(j)\left( \sigma^{\ast}(p-j)-4\sigma^{\ast}(\frac{p-j}{2}) \right).
 \]
 \end{corollary}
 \begin{proof}
 Immediate from Theorem~\ref{prime-2-squares}.
 \end{proof}
% \section{Sums of four squares}
 %
 \begin{theorem}\label{prime-4-8-squares}
 Let $p$ be an odd prime. Then
 \[(a)\
 \sum_{j=1}^{p-1}r_4(j) \left( \sigma^{\ast}(p-j)- 4\sigma^{\ast}(\frac{p-j}{2}) \right)= p^2 -1.
 \]
 \[ (b)\
 \sum_{j=1}^{p-1}r_8(j) \left( \sigma^{\ast}(p-j)- 4\sigma^{\ast}(\frac{p-j}{2}) \right)= p^4 -1.
 \]
 \end{theorem}
 \begin{proof}
 (a)\ As mentioned before $\sigma^{\ast}(p)=1+p$ and $\sigma^{\ast}(p/2)=0$. By
 identity (\ref{r2-r4-r8}) we have $r_4(p) = 8(1+p)$. Putting in formula~(\ref{four-squares}) we get
 \[
 8 p (1+p) = 8(p+1) + \sum_{j=1}^{p-1} r_4(j) \left(8 \sigma^{\ast}(p-j)-
    32 \sigma^{\ast}(\frac{p-j}{2}) \right),
 \]
 which gives the desired identity.
 \\
 (b)\ By Theorem~\ref{sum-of-squares}
 \begin{equation} \label{eight-squares}
 n r_8(n) = 16\left( \sigma^{\ast}(n)- 4 \sigma^{\ast}(n/2) +
     \sum_{j=1}^{n-1} r_8 (j) \left( \sigma^{\ast}(n-j)-
    4 \sigma^{\ast}(\frac{n-j}{2}) \right) \right).
 \end{equation}
 If $p$ is an odd prime, then by formulas (\ref{r2-r4-r8}) we have
 $r_8(p) = 16(1+ p^3)$, which combined with identity (\ref{eight-squares}) gives the result.
 \end{proof}
 \begin{corollary}
 Let $p$ be an odd prime. Then
 \[ (a)\
 \left( 1 + \sum_{j=1}^{p-1} r_2(j) \left(\sigma^{\ast}(p-j) - 4\sigma^{\ast}(\frac{p-j}{2}) \right) \right)^2
 =  
 \]
 \[
  1 + \sum_{j=1}^{p-1} r_4(j) \left(\sigma^{\ast}(p-j) - 4\sigma^{\ast}(\frac{p-j}{2})\right).
 \]
 \[ (b)\
  \left( 1 + \sum_{j=1}^{p-1} r_4(j) \left(\sigma^{\ast}(p-j) - 4\sigma^{\ast}(\frac{p-j}{2}) \right) \right)^2
 =  
 \]
 \[
 1 + \sum_{j=1}^{p-1} r_8(j) \left(\sigma^{\ast}(p-j) - 4\sigma^{\ast}(\frac{p-j}{2})\right).
 \]
 \end{corollary}
 \begin{proof}
 Immediate from Theorems \ref{prime-2-squares} and \ref{prime-4-8-squares}.
 \end{proof}
 \section{Sums of two, four, and six triangular numbers}
 The following formula for $t_2(n)$, $t_4(n)$, and $t_6(n)$ can be found in Ono, Robins, and Wahl~\cite{Ono-Robins-Wahl}
 and in Williams~\cite{Williams}.
 \begin{equation} \label{t2-t4-t6}
 \begin{split}
 t_2(n) &= \sum_{d\mid 4n+1}\left( \frac{-4}{d}\right), \\
 t_4(n) &= \sigma(2n+1), \\
 t_6(n) &= \frac{-1}{8} \sum_{d\mid 4n+3}\left( \frac{-4}{d}\right) d^2.
 \end{split}
 \end{equation}
 \begin{theorem} \label{p-4p+1}
 If $p$ and $4p+1$ are prime numbers, then
 \[
 \sum_{j=1}^{p-1} t_2(j)\left( \sigma(p-j)-4 \sigma(\frac{p-j}{2}) \right) = -1.
 \]
 \end{theorem}
 \begin{proof}
 Suppose that $p$ and $4p+1$ are primes. Clearly $\sigma(p)=1+p$, and by formulas (\ref{t2-t4-t6}) we have
 $t_2(p)= 2$. Further
 by Theorem~\ref{sum-of-triangulars}
 \begin{equation} \label{2-triangulars}
 p t_2(p) = 2 \sigma(p) - 8\sigma(p/2) +
    \sum_{j=1}^{p-1} t_2(j) \left( 2 \sigma(p-j) - 8\sigma(\frac{p-j}{2}) \right).
 \end{equation}
 Thus
 \[
 2p = 2(1+p) + 2 \sum_{j=1}^{p-1} t_2(j)\left( \sigma(p-j)- 4\sigma(\frac{p-j}{2}) \right).
 \]
 \end{proof}
 \begin{theorem}
 If $2n+1$ is prime, then
 \[
 \sum_{j=0}^{n-1} t_4(j) \left( \sigma(n-j) - 4\sigma(\frac{n-j}{2}) \right) = \frac{n(n+1)}{2}\quad (= \sum_{j=1}^n j).
 \]
 \end{theorem}
 \begin{proof}
 Assume that $2n+1$ is prime. Then by identities (\ref{t2-t4-t6}) we have
 $t_4(n) = 2n+2$.
 Moreover, from Theorem~\ref{sum-of-triangulars} we get
 \begin{equation} \label{4-triangulars}
 n t_4(n) =
   4 \sum_{j=0}^{n-1} t_4(j) \left( \sigma(n-j) - 4\sigma(\frac{n-j}{2}) \right).
 \end{equation}
 Combining these two identities proves the result.
 \end{proof}
 \begin{theorem}
 If $4n+3$ is prime, then
 \[
 \sum_{j=0}^{n-1} t_6(j) \left( \sigma(n-j)- 4\sigma(\frac{n-j}{2}) \right) = \frac{n(n+1)(2n+1)}{6}\quad (= \sum_{j=1}^n j^2).
 \]
 \end{theorem}
 \begin{proof}
 If $4n+3$ is prime, then by (\ref{t2-t4-t6}) we have
 \[
 t_6(n) = -\frac{1}{8}(1 - (4n+3)^2) = (n+1)(2n+1).
 \]
 Next by Theorem~\ref{sum-of-triangulars} we have
 \[
 n t_6(n) = 6 \sum_{j=0}^{n-1} t_6(j) \left( \sigma(n-j) - 4\sigma(\frac{n-j}{2}) \right).
 \]
  Using the previous two formulas we obtain the result.
 \end{proof}
 \section{Facts on some infinite products}
 Throughout this section we suppose that $x$ is a complex number such that $|x|<1$.
 \begin{theorem}\label{master-series}
 Let $a,\ b,\ n \in \mathbb{N}$ and let $I$ be a nonempty subset of 
 $\{ 0,1,2,\ldots,b-2\}$. Then the coefficients of the infinite product
 \[
 \prod_{n=1}^{\infty}\prod_{i\in I} (1-x^n)^{-a} (1-x^{bn-i})^{a}
 \]
 are all positive.
 \end{theorem}
 \begin{proof}
 Write 
 \[
 \prod_{n=1}^{\infty}\prod_{i\in I} (1-x^n)^{-a} (1-x^{bn-i})^{a} = \sum_{n=0}^{\infty} A(n) x^n.
 \]
 We show by induction that $A(n)=0$ for all $n\in\mathbb{N}_0$. Clearly $A(0)=1 >0$. Now assume that 
 the assertion is true for $j=0,1,\ldots, n-1$. By Theorem~\ref{main1} we have
 \[
 n A(n) = a \sigma(n) - a \sum_{i\in I}\sigma_{i,b}(n) +
 \sum_{j=1}^{n-1} A(j) \left( a\sigma(n-j) - a\sum_{i\in I}\sigma_{i,b}(n-j) \right).
 \]
 But
 \[
  a \sigma(n) - a \sum_{i\in I}\sigma_{i,b}(n) = 
  a \sum_{i=0}^{b-1} \sigma_{i,b}(n) - a \sum_{i\in I} \sigma_{i,b}(n) \geq
  a_{1,b}(n) >0.
  \]
 Thus by the induction hypothesis we have $n A(n) >0$ and therefore $A(n)>0$.
 \end{proof}
 \noindent
 For our next result we will need the following lemma.
 \begin{lemma}\cite[Exercise 22, p. 248]{Williams} \label{Williams-exercise}
 If $n\in \mathbb{N}$, then
 \[
 R(n)= 4 \sigma(n) - 4\sigma(n/2) + 8\sigma(n/4) - 32 \sigma(n/8) > 0.
 \]
 \end{lemma}
 \begin{proof}
 If $8\nmid n$, then
 \[
 R(n) = 4\sigma(n) - 4 \sigma(n/2) + 8 \sigma(n/4) \geq 4\sigma(n) - 4 \sigma(n/2) > 0.
 \]
 If $8 \mid n$, say $n=8k$, then repeatedly application of identities~(\ref{divisors-basics}) we obtain
 \[
 \begin{split}
 r(8k) 
 &= 
 4\sigma(8k) - 4\sigma(4k) + 8 \sigma(2k) -32 \sigma(k) \\
 &=
 4 \sigma^{o} (8k) + 4 \sigma^{o} (4k) + 16 \sigma^{o}(2k) \\
 & > 0.
 \end{split}
 \]
 This proves the lemma.
 \end{proof}
 \begin{theorem} \label{series-1}
 If
 \[ \prod_{n=1}^{\infty} (1+x^n)^4 (1+x^{2n})^2 (1+x^{4n})^4 = \sum_{n=0}^{\infty} \alpha(n) x^n, \]
 then $\alpha(n) > 0$ for all $n\in\mathbb{N}_0$.
 \end{theorem}
 \begin{proof}
 By induction on $n$. Clearly $\alpha(0)=1 >0$. Suppose now that the statement holds for $j=0,1,\ldots,n-1$.
 Note that
 \[
 \prod_{n=1}^{\infty} (1+x^n)^4 (1+x^{2n})^2 (1+x^{4n})^4 =
 \prod_{n=1}^{\infty} (1-x^n)^{-4} (1-x^{2n})^2 (1-x^{4n})^{-2} (1-x^{8n})^4,
 \]
 which by Theorem~\ref{main1} leads to
 \[
 n \alpha(n) = 4 \sigma(n)- 2\sigma^E (n) + 2 \sigma_{0,4}(n) - 4 \sigma_{0,8}(n) +
 \]
 \[
 \sum_{j=1}^{n-1} \alpha(j) \big( 4 \sigma(n-j)- 2\sigma^E (n-j) + 2 \sigma_{0,4}(n-j) - 4 \sigma_{0,8}(n-j) \big).
 \]
 Further, by (\ref{divisors-basics}) and Lemma~\ref{Williams-exercise} we get
 \[
 4 \sigma(n)- 2\sigma^E (n) + 2 \sigma_{0,4}(n) - 4 \sigma_{0,8}(n) =
 4 \sigma(n) - 4\sigma(n/2) + 8\sigma(n/4) - 32 \sigma(n/8) > 0.
 \]
 Then by the induction hypothesis we have $n \alpha(n) > 0$, and thus $\alpha(n) >0$.
 \end{proof}
 %
%%%%%%%%%%%%%%%%%%%%%%%%%%%%%%%%%%%%%%%%%%%%%%%%%%%%%%%%%%%%%%%%%%%%%%%%%%%%%%
%\noindent{\bf Acknowledgment.} The author is grateful to the referee
% for valuable comments and interesting suggestions.
%%%%%%%%%%%%%%%%%%%%%%%%%%%%%%%%%%%%%%%%%%%%%%%%%%%%%%%%%%%%%%%%%%%%%%%%%%%%%
 
%

\begin{thebibliography}{9}

%\bibitem{Alladi}
%  K. Alladi,
 % \emph{Partition Identities Involving Gaps and Weights},
 % Trans. Amer. Math. Soc. Volume 349, Number 12, (1997),
 % 5001-5019.

\bibitem{Apostol}
  T. M. Apostol,
  \emph{Introduction to Analytic Number Theory},
  Undergraduate Texts in Mathematics,
  Springer, 1 edition, 1976.

\bibitem{Dickson}
 L. E. Dickson,
 \emph{History of the Theory of Numbers},
 Vols. {I}-{III}, Chelsea Publ. Co., New York, 1952.
 
\bibitem{Elbachraoui}
 M. El Bachraoui,
 \emph{Inductive Formulas for some Arithmetic Functions}, Int. Journal of Number Theory,
 To appear (available on arXiv. 1103.5227).

%\bibitem{Ewell}
% John E. Ewell,
% \emph{Arithmetical consequences of a sextuple product identity},
% Rocky Mountain Journal of Mathematics, Volume 25, Number 4,
 (1995).

 %\bibitem{Ewell}
 %J. A. Ewell,
 %\emph{On an identity of Ramanujan},
 %Proc. Amer. Math. Soc., Volume 125, Number 12, (1997), 3769-3771.

 %\bibitem{Gordon-Ono}
 %B. Gordon and K. Ono,
 %\emph{Divisibility of certain partition functions by powers of primes},
 %The Ramanujan Journal, Volume 1, (1997), 25-34.

\bibitem{Hardy-Wright}
 G. H. Hardy and E. M. Wright,
 \emph{An Introduction to the Theory of Numbers},
 Oxford University Press, USA, 6th edition, 2008.

%\bibitem{Liouville}
 %J. Liouville,
 %\emph{Sur quelques formules g\'{e}n\'{e}rales qui peuvent \^{e}tre utiles dans la th\'{e}orie des nombres},
 %(7th article), J. Math. Pures Appl. 4, (1859), 1-8.

\bibitem{Ono-Robins-Wahl}
 K. Ono, S. Robins, and P. T. Wahl,
 \emph{On representation of integers as sums of triangular numbers},
 Aequationes Mathematicae, Volume 50, (1995), 73-94.

%\bibitem{Ramanujan}
% S. Ramanujan,
 %\emph{Collected papers},
 %Chelsea, New York, 1962.

\bibitem{Robbins}
 N. Robbins,
 \emph{Some identities connecting partition functions to other number theoretic functions},
 Rocky Mountain Journal of Mathematics, Volume 29, Number 1, (1999), 335-345.

\bibitem{Williams}
 K. S. Williams,
 \emph{Number Theory in the Spirit of Liouville},
 Cambridge University Press, New York, First edition, 2011.

\end{thebibliography}
\end{document}